\documentclass[11pt,twoside]{article}
\usepackage[latin1]{inputenc}
\usepackage{amsmath,amsfonts,amscd,amsthm}

\newtheorem{thm}{Théorème}[section]
\newtheorem{prop}[thm]{Proposition}
\newtheorem{lem}[thm]{Lemme}
\newtheorem{cor}[thm]{Corollaire}

\theoremstyle{definition}
\newtheorem{defin}[thm]{Définition}

\theoremstyle{remark}

\numberwithin{equation}{section}

\providecommand\ufootnote[1]{{\let\thefootnote\relax\footnote[0]{#1}}}

\newcommand{\cc}{\mathcal C}
\newcommand{\dc}{\mathcal D}
\newcommand{\ec}{\mathcal E}

\newcommand{\rb}{\mathbb R}
\newcommand{\nb}{\mathbb N}
\newcommand{\mb}{\mathbb M}
\newcommand{\cb}{\mathbb C}

\newcommand{\ci}{{\mathcal C}^\infty}

\newcommand{\ol}{\overline}

\newcommand{\pa}{\partial}
\newcommand{\opa}{\ol\partial}
\newcommand{\wt}{\widetilde}

 \DeclareMathOperator{\im}{Im}

\begin{document}

\title{ $L^p$- theory for the tangential Cauchy-Riemann equation}

\author{Christine LAURENT-THI\'{E}BAUT}

\date{}
\maketitle

\ufootnote{\hskip-0.6cm UJF-Grenoble 1, Institut Fourier, Grenoble,
F-38041, France\newline CNRS UMR 5582, Institut Fourier,
Saint-Martin d'H\`eres, F-38402, France
}

\ufootnote{\hskip-0.6cm  {\it 2010 A.M.S. Classification}~: 32V20, 32W10 .\newline {\it Key words}~: tangential Cauchy-Riemann equation, Serre duality, $L^p$ estimates.}

\bibliographystyle{amsplain}

Our purpose is to extend to the tangential Cauchy-Riemann operator the $L^p$ theory, for
$p\geq 1$, developed for the Cauchy-Riemann operator in \cite{LaLp}. We will replace complex manifolds by $s$-concave locally embeddable generic CR manifolds and the Cauchy-Riemann complex by the tangential Cauchy-Riemann complex.
We have to restrict ourself to locally embeddable CR manifolds since, up to now, very few is known in abstract CR manifolds outside the case of compact CR manifolds, where some results were obtained by C.D. Hill and M. Nacinovich \cite{HiNa1} and by M.-C. Shaw and L. Wang \cite{ShWa}.

Compare to the complex case some new difficulties arise in the CR
setting since the Poincaré lemma for the tangential Cauchy-Riemann
operator fails to hold in all degrees. In particular the Dolbeault isomorphism does not hold in all degrees and its classical proof works no more for large degrees.

We are particularly interested in solving the tangential
Cauchy-Riemann equation with exact support in $L^p$ spaces. We consider the following Cauchy problem.

Let  $M$  be a $\ci$-smooth, generic CR submanifold of real codimension $k$ in a complex manifold $X$ of complex dimension $n$ and $D$ a relatively compact domain in $M$ with $\ci$-smooth boundary. For $p\geq 1$, given any $(r,q)$-form $f$ with coefficients in $L^p(M)$,  $0\leq r\leq n$ and $1\leq q\leq n-k$, such that
$${\rm supp}~f\subset\ol D \quad {\rm and}\quad \opa_b f=0~{\rm in~the~weak~sense~in~}M,$$
does there exists a $(r,q-1)$-forme $g$  with coefficients in $L^p(M)$ such that
$${\rm supp}~g\subset\ol D \quad {\rm and}\quad \opa_b g=f~{\rm in~the~weak~sense~in~}M~?$$

We prove (see Corollary \ref{cauchy}) that the answer is yes as soon as $M$ is $s$-concave, $s\geq 2$, $D$ is the transversal intersection of $M$ with a relatively compact completely strictly pseudoconvex domain in $X$ with $\ci$-smooth boundary, $p>1$ and $1\leq q\leq s-1$.
\medskip

The paper is organized as follows. In section one we recall the main definitions of the theory of CR manifolds.

In the second section we develop the $L^p_{loc}$ theory, we prove
 a Poincaré lemma with $L^p$ estimates for the tangential Cauchy-Riemann operator and study the Dolbeault isomorphism. We prove
\begin{thm}
Let $M$ be a $\ci$-smooth, generic, $s$-concave, $s\geq 1$, CR submanifold of real codimension $k$ of a complex manifold $X$ of complex dimension $n$ and $p\geq1$ a real number.
If $0\leq r\leq n$, the natural maps
$$\Phi_q^p~:~H^{r,q}_\infty(M)\to H^{r,q}_{L^p_{loc}}(M)\quad {\rm and}\quad  \Psi_q^p~:~H^{r,q}_{L^p_{loc}}(M)\to  H^{r,q}_{cur}(M)$$
satisfy:

(i) $\Phi_q^p$ and $\Psi_q^p$ are surjective, if $0\leq q\leq s-1$ or $n-k-s+1\leq q\leq n-k$.

(ii) $\Phi_q^p$ and $\Psi_q^p$ are injective, if $0\leq q\leq s$ or $n-k-s+2\leq q\leq n-k$.

If $0\leq r\leq n$, the natural maps
$$\Phi_{c,q}^p~:~H^{r,q}_{c,\infty}(M)\to H^{r,q}_{c,L^p}(M)\quad {\rm and}\quad  \Psi_{c,q}^p~:~H^{r,q}_{c,L^p}(M)\to  H^{r,q}_{c,cur}(M)$$
satisfy:

(i) $\Phi_{c,q}^p$ and $\Psi_{c,q}^p$ are surjective, if $0\leq q\leq s-1$ or $n-k-s+1\leq q\leq n-k$.

(ii) $\Phi_{c,q}^p$ and $\Psi_{c,q}^p$ are injective, if $0\leq q\leq s$ or $n-k-s+2\leq q\leq n-k$.

\end{thm}
Then we extend the Serre duality introduced in \cite{LaLedualiteCR}
to the $L^p_{loc}$ tangential Cauchy-Riemann complexes, $p>1$, and get some
density results and a theorem on the solvability in $L^p$ spaces of the tangential Cauchy-Riemann equation with compact support in bidegree $(0,1)$.

The last section is devoted to the $L^p$ theory. First using local $L^p$ results proved in \cite {LaSh}, we develop an $L^p$ Andreotti-Grauert theory for large degrees. Then using once again Serre duality, we solve the weak Cauchy problem in $L^p$ for the tangential Cauchy-Riemann operator.

 I would like to thank E. Russ who pointed me out the reference \cite{Ho&al} to study the compacity of operators between $L^p$ spaces.

\section{CR manifolds}\label{s1}

Let $\mb$ be a $\ci$-smooth, paracompact differential manifold, we
denote by $T\mb$ the tangent bundle of $\mb$ and by $T_\cb
\mb=\cb\otimes T\mb$ the complexified tangent bundle.

\begin{defin}\label{structure}
An \emph {almost CR structure} on $\mb$ is a subbundle $H_{0,1}\mb$ of
$T_\cb \mb$ such that $H_{0,1}\mb\cap \ol{H_{0,1}\mb}=\{0\}$.

If the almost CR structure is integrable, i.e. for all
$Z,W\in\Gamma(\mb,H_{0,1}\mb)$ then $[Z,W]\in\Gamma(\mb,H_{0,1}\mb)$ , it is called a \emph {CR structure}.

If $H_{0,1}\mb$ is a CR structure, the pair $(\mb,H_{0,1}\mb)$ is
called an \emph {abstract CR manifold}.

The \emph{CR-dimension} of $\mb$ is defined by CR-dim
$\mb=\rm{rk}_\cb~H_{0,1}\mb$.
\end{defin}

We set $H_{1,0}\mb=\ol{H_{0,1}\mb}$ and $H\mb=H_{1,0}\mb\oplus H_{0,1}\mb$. For $X\in H\mb$, let $X^{0,1}$ denote the projection of $X$ on $H_{0,1}\mb$ and $X^{1,0}$ its projection on $H_{1,0}\mb$.
\medskip

From now on let $n$ and $k$ be integers such that the CR manifold $\mb$ is of real dimension $2n-k$ and CR-dimension $n-k$.

\begin{defin}\label{morphisme}
Let $(\mb,H_{0,1}\mb)$ and $(\mb',H_{0,1}\mb')$ be two CR manifolds and $F~:~\mb\to\mb'$ a $\cc^1$-map. The map $F$ is called a \emph{CR map} if and only if for each $x\in\mb$, $dF((H_{0,1}\mb)_x)\subset (H_{0,1}\mb')_{F(x)}$.
\end{defin}

In particular, if $(\mb,H_{0,1}\mb)$ is a CR manifold and $f$ a complex valued function, then $f$ is a CR function if and only if for any $\ol L\in H_{0,1}\mb$ we have $\ol L f=0$.
\medskip

We denote by $H^{0,1}\mb$ the dual bundle $(H_{0,1}\mb)^*$ of $H_{0,1}\mb$.
Let $\Lambda^{0,q}\mb=\bigwedge^q(H^{0,1}\mb)$, then
$\cc^s_{0,q}(\mb)=\Gamma^s(\mb,\Lambda^{0,q}\mb)$ is called the
space of $(0,q)$-forms of class $\cc^s$, $s\geq 0$ on $\mb$.

If the almost CR structure is a CR structure, i.e. if it is
integrable, and if $s\geq 1$, then we can define an operator
\begin{equation}
\opa_b~:~\cc^s_{0,q}(\mb)\to \cc^{s-1}_{0,q+1}(\mb)
\end{equation}
called the \emph{tangential Cauchy-Riemann operator} by setting
$\opa_b f=df_{|_{H_{0,1}\mb\times\dots\times H_{0,1}\mb}}$. It
satisfies $\opa_b\circ\opa_b=0$ and a complex valued function $f$ is a CR function if and only if $\opa_b f=0$

Let us consider the quotient bundle $\wt T\mb=T_\cb\mb/H_{0,1}\mb$ and
the canonical bundle map $\pi~:~T_\cb\mb\to \wt T\mb$.

For $u\in\Gamma^\infty(\mb,H_{1,0}\mb)$, we define the tangential Cauchy-Riemann operator by setting
$$\opa_b u(\ol L)=\pi[\ol L, W],$$
if $\ol L\in\Gamma(\mb,H_{0,1}\mb)$ and $W\in\Gamma^\infty(\mb,T_\cb\mb)$ satisfies $\pi(W)=u$. By the integrability of the CR structure, this does not depend on the choice of $W$ with $\pi(W)=u$.

Let us consider the vector bundle $\Lambda^{p,q}\mb=\Lambda^p(\wt T\mb)^*\otimes \Lambda^{0,q}\mb$ of $(p,q)$-forms on $\mb$ and set $\ec^{p,q}(\mb)=\Gamma^\infty(\mb,\Lambda^{p,q}\mb)$. We can define the tangential Cauchy-Riemann operator on $\ec^{p,q}(\mb)$ by setting
for $u\in\ec^{p,0}(\mb)$
$$(\opa_b u(\ol L))(L_1,\dots,L_p)=\ol L(u(L_1,\dots,L_p))-\sum_{i=1}^p u(L_1,\dots,L_{i-1},\ol L (L_i),L_{i+1},\dots,L_p),$$
for any $\ol L\in\Gamma(\mb,H_{0,1}\mb)$ and $L_i\in\Gamma(\mb,\wt T\mb)$, $1\leq i\leq p$, where $\ol L (L_i)=\opa_b L_i(\ol L)$ is defined as above, and
\begin{align*}
(\opa_b\theta(\ol L_1,\dots,\ol L_{q+1}))&(L_1,\dots,L_p)=\sum_{j=1}^{q+1}(-1)^{j+1}(\opa_b(\theta(\ol L_1,\dots,\widehat{\ol L_j},\dots,\ol L_{q+1}))(\ol L_j))(L_1,\dots,L_p)\\&+\sum_{j<k}(-1)^{j+k}\theta([\ol L_j,\ol L_k],\ol L_1,\dots,\widehat{\ol L_j},\dots,\widehat{\ol L_k},\dots,\ol L_{q+1})(L_1,\dots,L_p),
\end{align*}
if $\theta\in \ec^{p,q}(\mb)$, $\ol L_1,\dots,\ol L_{q+1}\in \Gamma(\mb,H_{0,1}\mb)$ and $L_i\in\Gamma(\wt T\mb)$, $1\leq i\leq p$.

So we get a differential complex $(\ec^{p,\bullet}(\mb),\opa_b)$ called the tangential Cauchy-Riemann complex for smooth forms. We denote by $H^{p,q}_\infty(\mb)$ the associated cohomology groups.
\medskip

Let $\dc^{p,q}(\mb)$ be the space of compactly supported elements of $\ec^{p,q}(\mb)$. We put on $\ec^{p,q}(\mb)$ the topology of uniform convergence on compact sets of the sections and all their derivatives. Endowed with this topology $\ec^{p,q}(\mb)$ is a Fréchet-Schwartz space.
Let $K$ be a compact subset of $\mb$ and $\dc_K^{p,q}(\mb)$ the closed subspace of $\ec^{p,q}(\mb)$ of forms with support in $K$ endowed with the induced topology. Choose $(K_n)_{n\in\nb}$ an exhausting sequence of compact subsets of $\mb$. Then $\dc^{p,q}(\mb)=\cup_{n\geq 0} \dc_{K_n}^{p,q}(\mb)$. We put on $\dc^{p,q}(\mb)$ the strict inductive limit topology defined by the FS-spaces $\dc_{K_n}^{p,q}(\mb)$. By restricting the tangential Cauchy-Riemann operator to $\dc^{p,q}(\mb)$ we get a new complex $(\dc^{p,\bullet}(\mb),\opa_b)$, whose cohomology groups are denoted $H^{p,q}_{c,\infty}(\mb)$.

The space of currents on $\mb$ of bidimension $(p,q)$ or bidegree $(n-p,n-k-q)$ is the dual of the space $\dc^{p,q}(\mb)$ and is denoted either by $\dc'_{p,q}(\mb)$, if we use the graduation given by the bidimension, or by $\dc'^{n-p,n-k-q}(\mb)$, if we use the graduation given by the bidegree. An element of $\dc'^{n-p,n-k-q}(\mb)$ can be identified with a distribution section of $\Lambda^{n-p,n-k-q}\mb$. We define the tangential Cauchy-Riemann operator on $\dc'^{n-p,n-k-q}(\mb)$ as the transpose map of the $\opa_b$ operator from $\dc^{p,q-1}(\mb)$ into $\dc^{p,q}(\mb)$, we still denote it by $\opa_b$ since its restriction to $\ec^{n-p,n-k-q}(\mb)$ coincides with the $\opa_b$ operator on smooth forms and we get a new complex $(\dc'^{n-p,\bullet}(\mb),\opa_b)$, the tangential Cauchy-Riemann complex for currents, which is the dual complex of the complex $(\dc^{p,\bullet}(\mb),\opa_b)$ (see section 2 in \cite{LaLedualiteCR}). The cohomology groups are denoted by $H^{n-p,n-k-q}_{cur}(\mb)$. Finally the dual of $\ec^{p,q}(\mb)$ denoted either by $\ec'_{p,q}(\mb)$ or $\ec'^{n-p,n-k-q}(\mb)$ is the space of currents on $\mb$ of bidimension $(p,q)$ or bidegree $(n-p,n-k-q)$ with compact support. Restricting the $\opa_b$ operator from $\dc'^{n-p,n-k-q}(\mb)$ to $\ec'^{n-p,n-k-q}(\mb)$, we get a fourth complex $(\ec'^{n-p,\bullet}(\mb),\opa_b)$, whose cohomology groups are denoted by $H^{n-p,n-k-q}_{c,cur}(\mb)$. This complex is the dual complex of the complex $(\ec^{p,\bullet}(\mb),\opa_b)$ (see section 2 in \cite{LaLedualiteCR}).

\medskip

The annihilator $H^0\mb$ of $H\mb=H_{1,0}\mb\oplus H_{0,1}\mb$ in
$T_\cb^*\mb$ is called the \emph{characteristic bundle} of $\mb$.
Given $p\in\mb$, $\omega\in H^0_p\mb$ and $X,Y\in H_p\mb$, we choose
$\wt\omega\in\Gamma(\mb,H^0\mb)$ and $\wt X,\wt Y\in
\Gamma(\mb,H\mb)$ with $\wt\omega_p=\omega$, $\wt X_p=X$ and $\wt
Y_p=Y$. Then $d\wt\omega(X,Y)=-\omega([\wt X,\wt Y])$. Therefore we can
associate to each $\omega\in H^0_p\mb$ an hermitian form
\begin{equation}\label{levi}
L_\omega(X)=-i\omega([\wt X,\ol{\wt X}])
\end{equation}
on $(H_{1,0})_p\mb$. This is called the \emph{Levi form} of $\mb$ at
$\omega\in H^0_p\mb$.

In the study of the $\opa_b$-complex two important geometric
conditions were introduced for CR manifolds of real dimension $2n-k$
and CR-dimension $n-k$. The first one by Kohn in the hypersurface
case, $k=1$, the condition Y(q), the second one by Henkin in
codimension $k$, $k\ge 1$, the $q$-concavity.

A CR manifold $\mb$ satisfies Kohn's condition $Y(q)$ at a point
$p\in\mb$ for some $0\le q\le n-1$, if the Levi form of $\mb$ at $p$
has at least $\max(n-q,q+1)$ eigenvalues of the same sign or at
least $\min(n-q,q+1)$ eigenvalues of opposite signs.

A CR manifold $\mb$ is said to be \emph{$q$-concave} at $p\in\mb$
for some $0\le q\le n-k$, if the Levi form $L_\omega$ at $\omega\in
H^0_p\mb$ has at least $q$ negative eigenvalues on $H_p\mb$ for
every nonzero $\omega\in H^0_p\mb$.

In \cite{ShWa} the condition Y(q) is extended to arbitrary
codimension.
\begin{defin}\label{y(q)}
An abstract CR manifold is said to satisfy \emph{condition Y(q)} for
some $1\le q\le n-k$ at $p\in\mb$ if the Levi form $L_\omega$ at
$\omega\in H^0_p\mb$ has at least $n-k-q+1$ positive eigenvalues or
at least $q+1$ negative eigenvalues on $H_p\mb$ for every nonzero
$\omega\in H^0_p\mb$.
\end{defin}

Note that in the hypersurface type case, i.e. $k=1$, this condition is
equivalent to the classical condition Y(q) of Kohn for
hypersurfaces and in particular if the CR structure is strictly peudoconvex,
i.e. the Levi form is positive definite or negative definite,
condition Y(q) holds for all $1\leq q<n-1$. Moreover, if $\mb$ is
$q$-concave at $p\in\mb$, then
$q\le (n-k)/2$ and condition Y(r) is satisfied at $p\in\mb$ for any
$0\le r\le q-1$ and $n-k-q+1\le r\le n-k$.

\begin{defin}\label{plongement}
Let $(\mb, H_{0,1}\mb)$ be an abstract CR manifold, $X$ be  a complex
manifold  and $F~:~\mb\to X$ be an embedding of class $\ci$, then
$F$ is called a \emph{CR embedding} if $dF(H_{0,1}\mb)$ is a subbundle
of the bundle $T_{0,1}X$ of the holomorphic vector fields on $X$ and
$dF(H_{0,1}\mb)=T_{0,1}X\cap T_\cb F(\mb)$.
\end{defin}

Let $F$ be a CR embedding of an abstract CR manifold into a complex
manifold $X$ and set $M=F(\mb)$, then $M$ is a CR manifold with the
CR structure $H_{0,1}M=T_{0,1}X\cap T_\cb M$.

Let $U$ be a coordinate domain in $X$, then
$F_{|_{F^{-1}(U)}}=(f_1,\dots ,f_N)$, with $N=\rm{dim}_\cb X$, and
$F$ is a CR embedding if and only if, for all $1\leq j\leq N$,
$\opa_b f_j=0$.

A CR embedding is called
\emph{generic} if $\rm{dim}_\cb X- \rm{rk}_\cb H_{0,1}M=\rm{codim}_\rb
M$.
\medskip

Most of the known results on CR manifolds concern CR manifolds which are locally embeddable at each point in $\cb^n$ and $s$-concave, $s\geq 1$. Moreover by a theorem of C.D. Hill and M. Nacinovich (\cite{HiNa1}, Proposition 3.1), each  $s$-concave, $s\geq 1$, CR manifold, which is locally embeddable at each point, can be generically embedded in a complex manifold. Consequently in most of the cases assuming that the abstract CR manifold $\mb$ is in fact a generic CR submanifold $M$ of a complex manifold $X$ is not a restriction.

\section{$L^p_{loc}$-cohomology for the tangential Cauchy-Riemann operator}

Let $\mb$ be a $\ci$-smooth CR manifold of real dimension $2n-k$ and CR-dimension $n-k$. We equip $\mb$ with an hermitian metric such that $H_{1,0}\mb~\bot~ H_{0,1}\mb$. If $p$ is a real number such that $p\geq 1$, let us consider the subspace $(L^p_{loc})^{r,q}(\mb)$ of $\dc'^{r,q}(\mb)$ consisting in the $L^p_{loc}$-sections of the vector bundle $\Lambda^{r,q}\mb$, $0\leq r\leq n$ and $0\leq q\leq n-k$, endowed with the topology of $L^p$ convergence on compact subsets of $\mb$. Taking the restriction to $(L^p_{loc})^{r,q}(\mb)$ of the tangential Cauchy-Riemann operator for currents we get an unbounded operator whose domain of definition is the set of forms $f$ with $L^p_{loc}$ coefficients such that $\opa_b f$ has also $L^p_{loc}$ coefficients and since $\opa_b\circ\opa_b=0$ we get a complex of unbounded operators $\big((L^p_{loc})^{r,\bullet}(\mb),\opa_b\big)$. We denote by $H^{r,q}_{L^p_{loc}}(\mb)$ the associated cohomology groups.

Let us consider the dual space of $(L^p_{loc})^{r,q}(\mb)$, when $1<p<\infty$. Since the injection of $\ec^{r,q}(\mb)$ into $(L^p_{loc})^{r,q}(\mb)$ has a dense image, the dual space of $(L^p_{loc})^{r,q}(\mb)$ is included in $\ec'^{n-r,n-k-q}(\mb)$ the space of $(n-r,n-k-q)$-currents with compact support on $\mb$, moreover by the Hölder inequality it coincides with the space $(L^{p'}_c)^{n-r,n-k-q}(\mb)$ of $L^{p'}$ forms with compact support in $\mb$ with $p'$ such that $\frac{1}{p}+\frac{1}{p'}=1$ and the duality pairing is defined by $<f,g>=\int_M f\wedge g$. Taking again the restriction to $(L^{p'}_c)^{r,q}(\mb)$ of the tangential Cauchy-Riemann operator for currents, we get a complex of unbounded operators $\big((L^{p'}_c)^{n-r,\bullet}(\mb),\opa_b\big)$. We denote by $H^{p,q}_{c,L^{p'}}(\mb)$ the associated cohomology groups. The complexes $\big((L^p_{loc})^{r,\bullet}(\mb),\opa_b\big)$ and $\big((L^{p'}_c)^{n-r,\bullet}(\mb),\opa_b\big)$, with $p'$ such that $\frac{1}{p}+\frac{1}{p'}=1$, are clearly dual to each other.

\subsection{ Dolbeault isomorphism}
Let $\mb$ be a $\ci$-smooth abstract CR manifold of real dimension $2n-k$, $k>0$, and CR-dimension $n-k$.

As in the complex case some relations between the smooth $\opa_b$-cohomology groups of $\mb$, the $\opa_b$-cohomology groups in the sense of currents of $\mb$ and the $L^p_{loc}$ $\opa_b$-cohomology groups of $\mb$ will result from a CR analog of the Dolbeault isomorphism. Let $\Omega^r\mb$ be the sheaf of CR $(r,0)$-forms, $0\leq r\leq n$ on $\mb$. If we can prove that the complexes $(\ec^{r,\bullet}(\mb),\opa_b)$, $(\dc'^{r,\bullet}(\mb),\opa_b)$ and $\big((L^p_{loc})^{r,\bullet}(\mb),\opa_b\big)$ are resolutions of the sheaf $\Omega^r\mb$, these groups will be isomorphic to the sheaf cohomology groups $H^\bullet(\mb,\Omega^r\mb)$. This will be true if $\mb$ is CR-hypoelliptic, which means that all CR $(r,0)$-forms on an open subset of $\mb$ are $\ci$-smooth and if the Poincaré lemma for the $\opa_b$ operator holds in the $\ci$-smooth category, the current category and the $L^p_{loc}$ category in degree $q$, $1\leq q\leq t(\mb)$, for some integer $t(\mb)$ depending on the geometry of $\mb$. In that case for $1\leq q\leq t(\mb)$, one gets
$$H^q(\mb,\Omega^r\mb)\sim H^{r,q}_\infty(\mb)\sim H^{r,q}_{cur}(\mb)\sim H^{r,q}_{L^p_{loc}}(\mb).$$

Unfortunately results on CR-hypoellipticity or on the validity of the Poincaré lemma for the $\opa_b$ operator  are only proven under the hypothesis that $\mb$ is locally embeddable and satisfies some concavity conditions. Therefore it comes from the remark at the end of section \ref{s1} that it is not a real restriction to assume that $\mb$ is generically embedded in a complex manifold.

From now on we assume that  $\mb=M$ is a $\ci$-smooth, generic CR submanifold of real codimension $k$ in a complex manifold $X$ of complex dimension $n$.

In that case the $\opa_b$ operator is CR-hypoelliptic as soon as $M$
is $1$-concave (cf. \cite{BaLa}, Corollaire 0.2) and the Poincaré
lemma holds for the $\opa_b$ operator in the smooth and in the current
categories in bidegree $(r,q)$, $0\leq r\leq n$ and $1\leq q\leq s-1$
or $n-k-s+1\leq q\leq n-k$, if $M$ is $s$-concave, $s\geq 2$ (cf. \cite{Hecras}, \cite{Nac}, \cite{NacVa}, \cite{BaLa}).

Moreover in this setting, the Poincaré lemma holds also in the $L^p_{loc}$ category,
\begin{lem}\label{poincareLp}
Let $M$ be a $\ci$-smooth, generic, $s$-concave, $s\geq 1$, CR submanifold of real codimension $k$ of a complex manifold $X$ of complex dimension $n$ and $p\geq 1$ a real number.
For each $x\in M$, there exists an open neighborhood $U$ of $x$ and an open neighborhood $V\subset\subset U$ of $x$ such that for any $\opa_b$-closed form $f\in(L^p_{loc})^{r,q}(U)$, $0\leq r\leq n$ and $1\leq q\leq s-1$ or $n-k-s+1\leq q\leq n-k$, there exists $g\in(L^p_{loc})^{r,q-1}(V)$ such that $f=\opa_b g$ on $V$.
\end{lem}
\begin{proof}
Let $U$ be an open neighborhood of $x$ such that the Bochner-Martinelli type formula from \cite{BaLa} holds. We get from Theorem 0.1 in \cite{BaLa} that there exists some continuous operators $\widehat R_M^{r,\wt q}$ from $\dc^{r,\wt q}(U)$ into $\ec^{r,\wt q-1}(U)$, $0\leq r\leq n$ and $1\leq\wt q\leq s$ or $n-k-s+1\leq\wt q\leq n-k$, such that, for any $\ci$-smooth $(r,q)$-form $u$ with compact support in $U$, $0\leq r\leq n$ and $1\leq q\leq s-1$ or $n-k-s+1\leq q\leq n-k$,
\begin{equation}\label{BM}
u=\opa_b \widehat R_M^{r,q} u+\widehat R_M^{r,q+1}\opa_b u.
\end{equation}
The operators $\widehat R_M^{r,q}$ are integral operators whose kernel $R_M^{r,q}$ are of weak type $\frac{2n}{2n-1}$ (cf. \cite{LaSh}, Lemma 4.2.2). Hence the operators $\widehat R_M^{r,q}$ extend as continuous operators from $(L^p_c)^{r,q}(U)$ into $(L^p_{loc})^{r,q-1}(U)$ and \eqref{BM} still holds if $u\in(L^p_c)^{r,q}(U)\cap Dom(\opa_b)$.

Let $f\in(L^p_{loc})^{r,q}(U)$ a $\opa_b$-closed form on $U$ and $\chi$ a $\ci$-smooth, positive function on $M$ with compact support in $U$ and constant equal to $1$ on some neighborhood $\wt V\subset\subset U$ of $X$, then
\begin{align*}
\chi f&=\opa_b \widehat R_M^{r,q} (\chi f)+\widehat R_M^{r,q+1}\opa_b (\chi f)\\
&=\opa_b \widehat R_M^{r,q} (\chi f)+\widehat R_M^{r,q+1}((\opa_b \chi)\wedge f).
\end{align*}
Since the singularities of the kernels $R_M^{r,q}$ are concentrated on the diagonal of $U\times U$ and the support of $\opa_b\chi$ is contained in $U\setminus \wt V$, the differential form $\widehat R_M^{r,q+1}((\opa_b \chi)\wedge f)$ is a smooth $\opa_b$-closed form on $\wt V$. By the Poincaré lemma for the $\opa_b$ operator for $\ci$-smooth forms there exists an open neighborhood $V\subset\subset \wt V$ of $x$ and a differential form $h\in\ec^{r,q-1}(V)$ such that $\widehat R_M^{r,q+1}((\opa_b \chi)\wedge f)=\opa_b h$. If we set $g=\widehat R_M^{r,q} (\chi f)+h$ then $g\in(L^p_{loc})^{r,q-1}(V)$ and $f=\opa_b g$ on $V$.
\end{proof}

\begin{thm}\label{isomplonge}
Let $M$ be a $\ci$-smooth, generic, $s$-concave, $s\geq 1$, CR submanifold of real codimension $k$ of a complex manifold $X$ of complex dimension $n$ and $p\geq 1$ a real number.

If $0\leq r\leq n$ and $0\leq q\leq s-1$, the natural maps
$$\Phi_q^p~:~H^{r,q}_\infty(M)\to H^{r,q}_{L^p_{loc}}(M)\quad {\rm and}\quad  \Psi_q^p~:~H^{r,q}_{L^p_{loc}}(M)\to  H^{r,q}_{cur}(M)$$
are isomorphisms, still called Dolbeault isomorphisms.

More precisely for all $0\leq r\leq n$ we have:

(i) If $1\leq q\leq s-1$, for any $f\in(L^p_{loc})^{r,q}(M)$ with $\opa_bf=0$ and any neighborhood $U$ of the support of $f$, there exists $g\in(L^p_{loc})^{r,q-1}(M)$ with support in $U$ such that $f-\opa_b g\in\ec^{r,q}(M)$.

(ii) If  $1\leq q\leq s$ and if $f\in(L^p_{loc})^{r,q}(M)$ satisfies $f=\opa_b S$ for some current $S\in\dc'^{r,q-1}(M)$, then, for any neighborhood $U$ of the support of $S$, there exists $g\in(L^p_{loc})^{r,q-1}(M)$ with support in $U$ such that $f=\opa_b g$ on $M$.
\end{thm}
\begin{proof}
For $0\leq q\leq s-1$, the proof works as in the complex case, see Propositions 1.1, 1.2 and 1.4 in \cite{LaLp}.

Let us consider the case $q=s$. Following the proof of Theorem 5.1 in
\cite{LaLedualiteCR}, it is sufficient to prove that for each $x\in
M$, there exists an open neighborhood $V$ of $x$ such that for any
$(r,s-1)$-current $T$ on $M$ with $\opa_b T\in(L^p_{loc})^{r,s}(M)$,
there exists $g\in(L^p_{loc})^{r,s-1}(V)$ such that $\opa_b T=\opa_b
g$ on $V$. Let $U$ be an open neighborhood of $x$ such that the
Bochner-Martinelli type formula from \cite{BaLa} holds. For $0\leq
r\leq n$ and $n-k-s+1\leq\wt q\leq n-k$, there exists some continuous
operators $\widehat R_M^{r\wt q}$ from $\dc^{r,\wt q}(U)$ into
$\ec^{r,\wt q-1}(U)$, such that, for any $\ci$-smooth $(r,\wt q)$-form $u$ with compact support in $U$,, $0\leq r\leq n$ and $n-k-s+1\leq\wt q\leq n-k$,
\begin{equation}\label{BMq}
u=\opa_b \widehat R_M^{r,\wt q} u+\widehat R_M^{r,\wt q+1}\opa_b u.
\end{equation}
Let us denote by $\wt R_M^{r,q}$ from $\ec'^{r,q}(U)$ into
$\dc'^{r,q-1}(U)$ the transpose of $\widehat R_M^{n-r,n-k-q+1}$,
$0\leq r\leq n$ and $1\leq q\leq s$, then by duality, since the $\opa_b$ operator for currents is the transpose of the $\opa_b$ operator for $\ci$-smooth forms, we have for any $(r,q)$-current $T$ with compact support in $U$, $0\leq r\leq n$ and $0\leq q\leq s-1$
\begin{equation}\label{BMcurant}
T=\opa_b \wt R_M^{r,q} T+\wt R_M^{r,q+1}\opa_b T.
\end{equation}
where we have set $\wt R_M^{r,0}=0$.
Moreover since, by the proof of Lemma \ref{poincareLp}, the operators
$\widehat R_M^{r,\wt q}$ extend as continuous operators from
$(L^p_c)^{r,\wt q}(U)$ into $(L^p_{loc})^{r,\wt q-1}(U)$ for any
$p\geq 1$, their transpose operators $\wt R_M^{r,q}$ define continuous
operators from $(L^p_c)^{r,q}(U)$ into $(L^p_{loc})^{r,q-1}(U)$ for
any $p>1$. For $p=1$ note that $\wt R^{r,q}_M$ is still an integral
operator whose kernel is of weak type $\frac{2n}{2n-1}$ hence
continuous from $(L^1_c)^{r,q}(U)$ into $(L^1_{loc})^{r,q-1}(U)$.

Assume $T$ is a $(r,s-1)$-current on $M$ with $\opa_b T\in(L^p_{loc})^{r,s}(M)$. Let $V\subset\subset U$ be an open neighborhood of $x$ and $\chi$ a $\ci$-smooth, positive function on $M$ with compact support in $U$ and constant equal to $1$ on some neighborhood of $V$, then we can apply \eqref{BMcurant} to $\chi T$ and take the $\opa_b$ of it, so we get
\begin{equation}
\opa_b(\chi T)=\opa_b(\wt R_M^{r,q+1}\opa_b(\chi T)).
\end{equation}
Note that $\opa_b(\chi T)=\chi\opa_b T+\opa_b(\chi)\wedge T$ with $\chi\opa_b T\in (L^p_{loc})^{r,s}(M)$, if $\opa_b T\in(L^p_{loc})^{r,s}(M)$, and support of $\opa_b(\chi)\wedge T$ is contained in $U\setminus V$. Since the operators are continuous from $(L^p_c)^{r,q}(U)$ into $(L^p_{loc})^{r,q-1}(U)$ and the singularity of their kernel is contained in the diagonal of $U\times U$, the restiction to $V$ of the current $\wt R_M^{r,q+1}\opa_b(\chi T)$ is a form $g\in(L^p_{loc})^{r,q-1}(V)$ such that
$\opa_b T=\opa_b g$ on $V$.
\end{proof}

\begin{cor}\label{isomcompact1}
Let $M$ be a $\ci$-smooth, connected, non compact, generic, $s$-concave, $s\geq 1$, CR submanifold of real codimension $k$ of a complex manifold $X$ of complex dimension $n$ and $p\geq 1$ a real number.
If $0\leq r\leq n$ and $0\leq q\leq s-1$, the natural maps
$$\Phi_{c,q}^p~:~H^{r,q}_{c,\infty}(M)\to H^{r,q}_{c,L^p}(M)\quad {\rm and}\quad  \Psi_{c,q}^p~:~H^{r,q}_{c,L^p_{loc}}(M)\to  H^{r,q}_{c,cur}(M)$$
are isomorphisms. Moreover $\Phi_{c,s}^p$ and $\Psi_{c,s}^p$ are only injective.

In particular $H^{r,q}_{c,\infty}(M)=0$ implies $H^{r,q}_{c,L^p}(M)=0$, when $0\leq r\leq n$ and $1\leq q\leq s-1$ and $H^{r,q}_{c,cur}(M)=0$ implies $H^{r,q}_{c,L^p}(M)=0$, when $0\leq r\leq n$ and $1\leq q\leq s$.
\end{cor}
\begin{proof}
First assume $0\leq q\leq s-1$. As in the complex case, since $H^{r,q}_\infty(M)\sim H^{r,q}_{cur}(M)$, the natural map $\Theta_{c,q}^p=\Psi_{c,q}^p\circ \Phi_{c,q}^p$ from $H^{r,q}_{c,\infty}(M)\to H^{r,q}_{c,cur}(M)$ is an isomorphism, consequently the map $\Phi_{c,q}^p$ is injective and the map $\Psi_{c,q}^p$ is surjective. Moreover by Theorem 5.1 in \cite{LaLedualiteCR} the map $\Theta_{c,s}^p$ is injective and so is $\Phi_{c,s}^p$.
On the other hand the assertion (i) of Theorem \ref{isomplonge} implies that $\Phi_{c,q}^p$ is surjective and the assertion (ii) of Theorem \ref{isomplonge} implies that $\Psi_{c,q}^p$ is injective. Moreover the last statement remains true if $q=s$.
\end{proof}

Unfortunately the Poincaré lemma for the $\opa_b$ operator fails to hold in any category in bidegree $(r,q)$ for $0\leq r\leq n$ and $s\leq q\leq n-k-s$, when $M$ is $s$-concave, hence the classical cohomological methods can no more be applied to get the Dolbeault isomorphisms when $q\geq n-k-s+1$, even if the Poincaré lemma holds for such $q$. The method developed in \cite{LaLeDolbeault} cannot either be used to get the Dolbeault isomorphisms with the $L^p_{loc}$ cohomology groups in high degrees since the notion of trace is not well defined in $L^p$. We will use the regularization method introduced by S. Sambou in \cite{Sareg}, which depends on the fundamental solution for the $\opa_b$ operator constructed in \cite{BaLa}.

\begin{thm}\label{isomgd}
Let $M$ be a $\ci$-smooth, generic, $s$-concave, $s\geq 1$, CR
submanifold of real codimension $k$ of a complex manifold $X$ of
complex dimension $n$ and $p\geq 1$ a real number.
If $0\leq r\leq n$, the natural maps
$$\Phi_q^p~:~H^{r,q}_\infty(M)\to H^{r,q}_{L^p_{loc}}(M)\quad {\rm and}\quad  \Psi_q^p~:~H^{r,q}_{L^p_{loc}}(M)\to  H^{r,q}_{cur}(M)$$
satisfy:

(i) $\Phi_q^p$ and $\Psi_q^p$ are surjective, if $n-k-s+1\leq q\leq n-k$.

(ii) $\Phi_q^p$ and $\Psi_q^p$ are injective, if $n-k-s+2\leq q\leq n-k$.

In particular $H^{r,q}_{\infty}(M)=0$ implies $H^{r,q}_{L^p_{loc}}(M)=0$, when $0\leq r\leq n$ and $n-k-s+1\leq q\leq n-k$ and $H^{r,q}_{cur}(M)=0$ implies $H^{r,q}_{L^p_{loc}}(M)=0$, when $0\leq r\leq n$ and $n-k-s+2\leq q\leq n-k$.
\end{thm}
\begin{proof}
In \cite{Sareg}, S. Sambou proved that for any  $0\leq r\leq n$ and $1\leq q\leq s-1$ or $n-k-s+1\leq q\leq n-k$ there exists linear operators
$$R_{r,q}~:~\dc'^{r,q}(M)\to \ec^{r,q}(M)\quad {\rm and}\quad A_{r,q}~:~\dc'^{r,q}(M)\to \dc'^{r,q-1}(M)$$
such that $A_{r,q}$ has the same regularity properties than the
operators $\wt R_M^{r,q}$ defined in the proof of Theorem
\ref{isomplonge}, in particular $A_{r,q}$ is continuous from
$(L^p_{loc})^{r,q}(M)$ into $(L^p_{loc})^{r,q-1}(M)$ for $p\geq 1$ and from  $\ec^{r,q}(M)$ into $\ec^{r,q-1}(M)$, and, for any $T\in \dc'^{r,q}(M)$,
$$T=R_{r,q}T+A_{r,q+1}\opa_b T+\opa_b A_{r,q}T,$$
if $1\leq q\leq s-2$ or $n-k-s+1\leq q\leq n-k$ and
$$T=R_{r,q}T+\opa_b A_{r,q}T,$$
if $q=s-1$ and $ \opa_b T=0$.

It follows that for $f\in(L^p_{loc})^{r,q}(M)$ with $\opa_b f=0$ we get $f=R_{r,q}f+\opa_b A_{r,q}f$ where $R_{r,q}f$ is $\ci$-smooth on $M$ and $ A_{r,q}f$ is $L^p_{loc}$. Hence $\Phi_q^p$ is surjective for $1\leq q\leq s-1$ or $n-k-s+1\leq q\leq n-k$. In the same way we get that $\Psi_q^p$ is surjective if $1\leq q\leq s-1$ or $n-k-s+1\leq q\leq n-k$, since $\ec^{r,q}(M)\subset (L^p_{loc})^{r,q}(M)$.

If $f\in(L^p_{loc})^{r,q}(M)$ satisfies $f=\opa_b T$, where $T$ is a $(r,q-1)$-current on $M$, then $f=\opa_b (R_{r,q-1}T+A_{r,q}f)$ for $1\leq q\leq s$ or $n-k-s+2\leq q\leq n-k$ and $R_{r,q-1}T+A_{r,q}f\in(L^p_{loc})^{r,q}(M)$. Hence $\Psi_q^p$ is injective, if $1\leq q\leq s$ or $n-k-s+2\leq q\leq n-k$. The proof of the injectivity of $\Phi_q^p$ uses the same arguments since the operators $A_{r,q}$ is continuous from $\ec^{r,q}(M)$ into $\ec^{r,q-1}(M)$.
\end{proof}

Note that the proof of Theorem \ref{isomgd} gives also an alternative proof for the Dolbeault isomorphism in small degrees.

As in the complex case, it follows from Theorem \ref{isomgd} that the assertion (ii) from Theorem \ref{isomplonge} holds also for $n-k-s+2\leq q\leq n-k$.
\begin{cor}\label{Lpcompact}
Let $M$ be a $\ci$-smooth, generic, $s$-concave, $s\geq 1$, CR
submanifold of real codimension $k$ of a complex manifold $X$ of
complex dimension $n$ and $p\geq 1$ a real number. Let $(r,q)$ such that $0\leq r\leq n$ and $1\leq q\leq s$ or $n-k-s+2\leq q\leq n-k$, and $f\in(L^p_c)^{r,q}(M)$ with $f=\opa_b S$ for some current $S\in\dc'^{r,q-1}(M)$, then, for any neighborhood $U$ of the support of $S$, there exists $g\in(L^p_c)^{r,q-1}(M)$ with $\opa_b g=f$ and ${\rm supp}~g\subset U$.
\end{cor}

Moreover the Dolbeault isomorphism holds also in large degrees for the cohomology groups with compact support.
\begin{cor}\label{isomcompact2}
Let $M$ be a $\ci$-smooth, generic, $s$-concave, $s\geq 1$, CR submanifold of real codimension $k$ of a complex manifold $X$ of complex dimension $n$ and $p>1$ a real number.
If $0\leq r\leq n$ and $n-k-s+2\leq q\leq n-k$, the natural maps
$$\Phi_{c,q}^p~:~H^{r,q}_{c,\infty}(M)\to H^{r,q}_{c,L^p}(M)\quad {\rm and}\quad  \Psi_{c,q}^p~:~H^{r,q}_{c,L^p}(M)\to  H^{r,q}_{c,cur}(M)$$
are isomorphisms. Moreover $\Phi_{c,n-k-s+1}^p$ and $\Psi_{c,n-k-s+1}^p$ are only surjective.

In particular $H^{r,q}_{c,\infty}(M)=0$ implies $H^{r,q}_{c,L^p}(M)=0$, when $0\leq r\leq n$ and $n-k-s+1\leq q\leq n-k$ and $H^{r,q}_{c,cur}(M)=0$ implies $H^{r,q}_{c,L^p}(M)=0$, when $0\leq r\leq n$ and $n-k-s+2\leq q\leq n-k$ .
\end{cor}

To end this section we will precise the support of the solution of the tangential Cauchy-Riemann equation $\opa_b g=f$, when $f$ is a $(r,1)$-form, $0\leq r\leq n$, with compact support in $M$ and $L^p$ coefficients.

\begin{prop}\label{d1}
Let $M$ be a $\ci$-smooth, connected, generic, $1$-concave CR
submanifold of real codimension $k$ of a complex manifold $X$ of
complex dimension $n$, $D$ a relatively compact subset of $M$ and
$p\geq 1$ a real number. Assume
$H^{r,1}_{c,cur}(M)=0$ for $0\leq r\leq n$, and $M\setminus D$ is connected, then for any $\opa_b$-closed $(r,1)$-form $f$ with support in $\ol D$ and $L^p$ coefficients there exists a unique $(r,0)$-form $g$ with $L^p$ coefficients and such that $\opa_b g=f$ and ${\rm supp}~g\subset\ol D$.
\end{prop}
\begin{proof}
Since $M$ is $1$-concave, it follows from Corollary \ref{isomcompact1} that for any $0\leq r\leq n$, $H^{r,1}_{c,cur}(M)=0$ implies $H^{r,1}_{c,L^p}(M)=0$. Hence there exists an $(r,0)$-form $g$ with compact support in $M$ and $L^p$ coefficients such that $\opa_b g=f$. The form $g$ vanishes on some open subset of $M\setminus D$ and is CR on $M\setminus {\rm supp}~f$, which contains $M\setminus\ol D$. By analytic continuation, which holds in $1$-concave manifolds, $g$ vanishes on $M\setminus\ol D$, since $M\setminus D$ is connected, therefore ${\rm supp}~g\subset\ol D$. Let $h$ be another solution with compact support in $M$, then $g-h$ is CR with compact support in $M$, hence vanishes on $M$ since $M$ is connected and $g=h$.
\end{proof}

\subsection{Serre duality}

In this section $M$ will always be a $\ci$-smooth, generic CR submanifold of real codimension $k$ in a $n$-dimensional complex manifold $X$.

As we already mentioned in the previous section, if $p$ and $p'$ satisfy $\frac{1}{p}+\frac{1}{p'}=1$, the spaces $(L^p_{loc})^{r,q}(M)$ and $(L^{p'}_c)^{n-r,n-k-q}(M)$ are dual from each other, with the duality pairing $<f,g>=\int_M f\wedge g$ and we get two dual complexes $\big((L^p_{loc})^{r,\bullet}(M),\opa_b\big)$
and $\big((L^{p'}_c)^{n-r,\bullet}(M),\opa_b\big)$ for each $0\leq r\leq n$.

Fix some integer $r$, $0\leq r\leq n$. We can apply Theorem 1.5 in \cite{LaLedualiteCR} to the complexes $(\dc^{r,\bullet},\opa_b)$, $(\ec'^{r,\bullet},\opa_b)$ and $\big((L^{p'}_c)^{n-r,\bullet}(M),\opa_b\big)$. We deduce that if $\frac{1}{p}+\frac{1}{p'}=1$, for any $0\leq q\leq n-k$, the natural maps
\begin{align*}
^\sigma H^{r,q}_{cur}(M)&\to \big(H^{n-r,n-k-q}_{c,\infty}(M)\big)'\\
^\sigma H^{r,q}_{\infty}(M)&\to \big(H^{n-r,n-k-q}_{c,cur}(M)\big)'\\
^\sigma H^{r,q}_{c,cur}(M)&\to \big(H^{n-r,n-k-q}_{\infty}(M)\big)'\\
^\sigma H^{r,q}_{c,\infty}(M)&\to \big(H^{n-r,n-k-q}_{cur}(M)\big)'\\
^\sigma H^{r,q}_{L^p_{loc}}(M)&\to \big(H^{n-r,n-k-q}_{c,L^{p'}}(M)\big)'\\
^\sigma H^{r,q}_{c,L^p}(M)&\to \big(H^{n-r,n-k-q}_{L^{p'}_{loc}}(M)\big)'.
\end{align*}
are isomorphisms.

Moreover, from Theorem 3.2 in \cite{LaLedualiteCR}, we get that $H^{n-r,n-k-q}_{c,cur}(M)$ is Hausdorff if and only if  $H^{r,q+1}_{\infty}(M)$ is Hausdorff and if these assertions hold, then
\begin{align*}
&H^{n-r,n-k-q}_{c,cur}(M)\sim \big(H^{r,q}_{\infty}(M)\big)'\\
&H^{r,q+1}_{\infty}(M)\sim \big(H^{n-r,n-k-q-1}_{c,cur}(M)\big)'.
\end{align*}

We deduce easily that, if there exists an integer $s(M)\geq 1$ such that $H^{r,q}_{\infty}(M)$ is Hausdorff for $0\leq q\leq s(M)-1$ and $n-k-s(M)+1\leq q\leq n-k$, then
$$H^{r,q}_{c,cur}(M) ~{\rm is~ Hausdorff~ and}~ H^{r,q}_{c,cur}(M)\sim \big(H^{n-r,n-k-q}_{\infty}(M)\big)',$$
if $1\leq q\leq s(M)$ and $n-k-s(M)+2\leq q\leq n-k$. In particular
$$H^{r,q}_{\infty}(M)=0\quad {\rm for}\quad  0\leq q\leq s(M)-1~ {\rm and}~ n-k-s(M)+1\leq q\leq n-k$$
implies
$$H^{r,q}_{c,cur}(M)=0\quad {\rm for}\quad  0\leq q\leq s(M)-1~ {\rm and}~ n-k-s(M)+2\leq q\leq n-k.$$
Associated with the Corollaries \ref{isomcompact1} and \ref{isomcompact2}, this gives
\begin{prop}\label{annulcomp}
Let $M$ be a $\ci$-smooth, generic, $s$-concave, $s\geq 1$, CR submanifold of real codimension $k$ in a $n$-dimensional complex manifold $X$  and $p>1$ a real number. Assume that  $H^{r,q}_{\infty}(M)=0$ for $0\leq r\leq n$ and $0\leq q\leq s-1$ or $n-k-s+1\leq q\leq n-k$, then $H^{r,q}_{c,L^p}(M)=0$ for $0\leq r\leq n$ and $0\leq q\leq s-1$ or $n-k-s+2\leq q\leq n-k$.
\end{prop}

 We will use the following notations:
\begin{align*}
&Z^{r,q}_\infty(M)=\ec^{r,q}(M)\cap\ker\opa_b,\quad &E^{r,q}_\infty(M)=\ec^{r,q}(M)\cap\im\opa_b\\
&Z^{r,q}_{c,\infty}(M)=\dc^{r,q}(M)\cap\ker\opa_b,\quad &E^{r,q}_{c,\infty}(M)=\dc^{r,q}(M)\cap\im\opa_b\\
&Z^{r,q}_{cur}(M)=\dc'^{r,q}(M)\cap\ker\opa_b,\quad &E^{r,q}_{cur}(M)=\dc'^{r,q}(M)\cap\im\opa_b\\
&Z^{r,q}_{c,cur}(M)=\ec'^{r,q}(M)\cap\ker\opa_b,\quad &E^{r,q}_{c,cur}(M)=\ec'^{r,q}(M)\cap\im\opa_b\\
&Z^{r,q}_{L^p_{loc}}(M)=(L^p_{loc})^{r,q}(M)\cap\ker\opa_b,\quad &E^{r,q}_{L^p_{loc}}(M)=(L^p_{loc})^{r,q}(M)\cap\im\opa_b\\
&Z^{r,q}_{c,L^p}(M)=(L^p_c)^{r,q}(M)\cap\ker\opa_b,\quad &E^{r,q}_{c,L^p}(M)=(L^p_c)^{r,q}(M)\cap\im\opa_b
\end{align*}

In the $L^p$ setting the Serre duality can be expressed in the following way:
\begin{thm}\label{serre}
Let $M$ be a $\ci$-smooth, generic CR submanifold of real codimension $k$ in a $n$-dimensional complex manifold $X$. Let $r,q\in\nb$ with $0\leq r\leq n$ and $0\leq q\leq n-k$, then for $p,p'\in\rb$ with $\frac{1}{p}+\frac{1}{p'}=1$ the following assertions are equivalent:

(i) $E^{n-r,n-k-q}_{c,L^p}(M)=\{g\in (L^p_c)^{n-r,n-k-q}(M)~|~<g,f>=0, \forall f\in Z^{r,q}_{L^{p'}_{loc}}(M)\}$

(ii) $H^{n-r,n-k-q}_{c,L^p}(M)$ is Hausdorff

(iii) $E^{r,q+1}_{L^{p'}_{loc}}(M)=\{g\in (L^{p'}_{loc})^{r,q+1}(M)~|~<g,f>=0, \forall f\in Z^{n-r,n-k-q-1}_{c,L^p}(M)\}$

(iv) $H^{r,q+1}_{L^{p'}_{loc}}(M)$ is Hausdorff.

{\parindent=0pt If these assertions hold, then}
\begin{align*}
&H^{n-r,n-k-q}_{c,L^p}(M)\sim \big(H^{r,q}_{L^{p'}_{loc}}(M)\big)'\\
&H^{r,q+1}_{L^p_{loc}}(M)\sim \big(H^{n-r,n-k-q-1}_{c,L^{p'}}(M)\big)'.
\end{align*}

\end{thm}
\begin{proof}
Since for each $0\leq r\leq n$ the complexes $\big((L^p_{loc})^{r,\bullet}(M),\opa_b\big)$
and $\big((L^{p'}_c)^{n-r,\bullet}(M),\opa_b\big)$ are dual complexes and $\big((L^p_{loc})^{r,\bullet}(M),\opa_b\big)$ is a complex of Fréchet-Schwartz spaces, we can apply Theorem 1.6 of \cite{LaLedualite}.
\end{proof}

As in \cite{LaLedualite}, we will use Serre duality to get some density theorems for the spaces of $L^p$ forms.

\begin{thm}\label{densite1}
Let $M$ be a $\ci$-smooth, connected, non compact, generic, $1$-concave CR submanifold of real codimension $k$ in a $n$-dimensional complex manifold $X$  and $p>1$ a real number. For any integer $r$, $0\leq r\leq n$, the space $Z^{r,n-k-1}_\infty(M)$ of $\ci$-smooth, $\opa_b$-closed, $(r,n-k-1)$-forms is dense in the space $Z^{r,n-k-1}_{L^p_{loc}}(M)$ of $\opa_b$-closed, $(r,n-k-1)$-forms with $L^p_{loc}$ coefficients for the $L^p_{loc}$ topology.
\end{thm}
\begin{proof}
By the Hahn-Banach theorem, it is sufficient to prove that for any $g\in (L^{p'}_c)^{n-r,1}(M)$ such that $<g,\varphi>=0$ for all $\varphi\in Z^{r,n-k-1}_\infty(M)$, we have $<g,f>=0$ for all $f\in Z^{r,n-k-1}_{L^p_{loc}}(M)$. Let
$$g\in (L^{p'}_c)^{n-r,1}(M)\cap \{T\in\ec'^{n-r,1}(M)~|~<T,\varphi>=0,~  \forall \varphi\in Z^{r,n-k-1}_\infty(M)\}.$$
 Malgrange's theorem \cite{LaLeMalgrange} claims that, under the assumptions of the theorem, $H^{r,n-k}_\infty(M)=0$, then by Serre duality (see Theorem 3.2 in \cite{LaLedualite}) we get that $g\in (L^{p'}_c)^{n-r,1}(M)\cap E^{n-r,1}_{c,cur}(M)$. From assertion (ii) in Theorem \ref{isomplonge}, we deduce that there exists $h\in (L^{p'}_c)^{n-r,0}(M)$ such that $g=\opa_b h$. Let $f\in Z^{r,n-k-1}_{L^p_{loc}}(M)$, then $<g,f>=<\opa_b h,f>=<h,\opa_b f>=0$.
\end{proof}

Adding the assumption that some cohomology group is Hausdorff to replace Malgrange's theorem the previous theorem can be extended to $s$-concave CR manifold, $s\geq 1$.

 \begin{thm}
Let $M$ be a $\ci$-smooth, generic, $s$-concave, $1\leq s\leq n-k$, CR submanifold of real codimension $k$ in an $n$-dimensional complex manifold $X$  and $p>1$ a real number. For any integers $r,q$, $0\leq r\leq n$ and $1\leq q\leq s$, assume that $H^{r,n-k-q+1}_\infty(M)$ is Hausdorff, then the space $Z^{r,n-k-q}_\infty(M)$ of $\ci$-smooth, $\opa_b$-closed, $(r,n-k-q)$-forms is dense in the space $Z^{r,n-k-q}_{L^p_{loc}}(M)$ of $\opa_b$-closed, $(r,n-k-q)$-forms with $L^p_{loc}$ coefficients for the $L^p_{loc}$ topology.
\end{thm}

Let us give some geometrical conditions on $M$ which ensure that, for all $0\leq r\leq n$, $H^{r,n-k-q+1}_\infty(M)$ is Hausdorff for some $1\leq q\leq n-k$.

1) If $M$ is compact and $s$-concave, by Theorem 4 in \cite{Hecras} or Theorem 4.1 in \cite{HiNa1}, $H^{r,n-k-q+1}_\infty(M)$ is finite dimensional hence Hausdorff for $1\leq q\leq s$ and $n-k-s+2\leq q\leq n-k$ and Hausdorff for $q=n-k-s+1$.

2) If $M$ is connected, non compact and $1$-concave, then by Theorem 0.1 in \cite{LaLeMalgrange}, $H^{r,n-k}_\infty(M)=0$.

3) The cohomology groups $H^{r,n-k-q+1}_\infty(M)$ are also finite dimensional hence Hausdorff for $1\leq q\leq s$, as soon as $M$ is $s$-concave and admits an exhausting function with good pseudoconvexity properties (see for example Theorem 6.1 in \cite{HiNa1}). This holds in particular  when $M$ of real dimension $2n-k$ is generically embedded in an n-dimensional pseudoconvex complex manifold $X$.

Let us end by a separation theorem for the $L^p$ cohomology with compact support in bidegree $(r,1)$, $0\leq r\leq n$.

\begin{thm}\label{resolcomp1}
Let $M$ be a $\ci$-smooth, connected, non compact, generic, $1$-concave CR submanifold of real codimension $k$ in a $n$-dimensional complex manifold $X$  and $p>1$ a real number. For any integer $r$, $0\leq r\leq n$, the cohomology groups $H^{r,1}_{c,L^p}(M)$ are Hausdorff. More precisely, if $f\in (L^p_c)^{r,1}(M)$ is a $\opa_b$-closed form such that $<f,\varphi>=0$ for all $\opa_b$-closed, $\ci$-smooth $(n-r,n-k-1)$-form $\varphi$ in a neighborhood of the support of $f$, there exists a unique $g\in (L^p_c)^{r,0}(M)$ such that $\opa_b g=f$ and if the support of $f$ is contained in the closure of some relatively compact domain $D$ of $M$ then ${\rm supp}~g\subset\ol D$.
\end{thm}
\begin{proof}
Under the hypotheses of the theorem,  by Theorem 0.1 in \cite{LaLeMalgrange}, $H^{n-r,n-k}_\infty(M)=0$, for all $0\leq r\leq n$. Using Theorem \ref{isomgd}, we get $H^{n-r,n-k}_{L^p_{loc}}(M)=0$ for any $p>1$ and Serre duality (Theorem \ref{serre}) implies that the cohomology groups $H^{r,1}_{c,L^p}(M)$ are Hausdorff for all $0\leq r\leq n$ and any $p>1$. In the last assertion the existence of $g$ follows from (i) in Theorem \ref{serre} and from the density of  $Z^{n-r,n-k-1}_\infty(M)$ in $Z^{n-r,n-k-1}_{L^p_{loc}}(M)$ for the $L^p_{loc}$ topology proved in Theorem \ref{densite1}. The uniqueness of $g$ results from analytic continuation in connected $1$-concave CR manifolds and of the non compactness of $M$. To get the support condition let us replace the CR manifold $M$ by $\wt M=M\setminus\{z_1,\dots,z_l\}$, where the points $z_1,\dots,z_l$ are chosen such that $z_i\in D_i$, $i=1,\dots,l$, and the $D_i$'s, $i=1,\dots,l$, are the relatively compact connected components of $M\setminus D$. The new manifold $\wt M$ has the same properties than $M$ and we get that in fact the support of $g$ is a compact subset of $\wt M$, moreover $g$ is CR in $M\setminus\ol D$ and vanishes in a neighborhood of each $z_i$, $i=1,\dots,l$, therefore by analytic continuation it vanishes on $M\setminus\ol D$.
\end{proof}

Note that, if the orthogonality condition is only satisfied for the forms $\varphi\in Z^{n-r,n-k-1}_\infty(M)$ we have to assume as in Proposition \ref{d1}
that $M\setminus D$ is connected to get that the support of $g$ is
contained in $\ol D$.

\section{$L^p$-theory for the tangential Cauchy-Riemann operator}

Throughout this section $M$ will always be a $\ci$-smooth, generic CR submanifold of real codimension $k$ in a complex manifold $X$ of complex dimension $n$ and $D$ a relatively compact domain in $M$ with rectifiable boundary. We equip $X$ with an hermitian metric such that $T_{1,0}X~\bot~ T_{0,1}X$ and we consider on $M$ the induced metric then $H_{1,0}M~\bot~ H_{0,1}M$. We denote by $dV$ the volume form on $M$ associated to that metric.

For any real number $p\geq 1$, we define the space $L^p(D)$ as the set of all functions $f$ on $D$ with complex values such that $|f|^p$ is integrable on $D$.
Since $D\subset\subset X$, we have
$$\dc(D)\subset\ec(X)_{|_D}\subset L^p(D),$$
where $\ec(X)_{|_D}$ is the space of restrictions to $D$ of $\ci$-smooth functions on $X$ and $\dc(D)$ the subspace of $\ci$-smooth functions with compact support in $D$.
The space $L^p(D)$ is a Banach space, it is the completed space of $\dc(D)$ for the $\|.\|_p$-norm  defined by
$$\|f\|_p=(\int_D |f|^pdV)^\frac{1}{p}.$$
For  any fixed $r$, $0\leq r\leq n$, we consider the spaces $L^p_{r,q}(D)$ of $(r,q)$-forms, $0\leq q\leq n-k$,  with $L^p$ coefficients and we are going to define unbounded operators from $L^p_{r,q}(D)$ into $L^p_{r,q+1}(D)$ whose restriction to the space $\dc^{r,q}(D)$ will coincide with the tangential Cauchy-Riemann operator assocated to the CR structure on $M$.

Let us define $\opa_{b,c}~:~L^p_{r,q}(D)\to L^p_{r,q+1}(D)$ as the closed minimal extension of ${\opa_b}_{|_{\dc^{r,q}(D)}}$, it is a closed operator by definition and
$f\in Dom(\opa_{b,c})$ if and only if there exists a sequence $(f_\nu)_{\nu\in\nb}$ of forms in $\dc^{r,q}(D)$ and a form $g\in\dc^{r,q+1}(D)$ such that $(f_\nu)_{\nu\in\nb}$ converges to $f$ in $L^p_{r,q}(D)$ and $(\opa_b f_\nu)_{\nu\in\nb}$  converges to $g$ in $L^p_{r,q+1}(D)$, and we have $g=\opa_{b,c} f$.

We can also consider $\opa_{b,s}~:~L^p_{r,q}(D)\to L^p_{r,q+1}(D)$, the closed minimal extension of ${\opa_b}_{|_{\ec^{r,q}(M)_{|_D}}}$, it is also a closed operator and $f\in Dom(\opa_{b,s})$ if and only if there exists a sequence $(f_\nu)_{\nu\in\nb}$ of forms in $\ec^{r,q}(M)$ and a form $g\in\ec^{r,q+1}(M)$ such that $(f_\nu)_{\nu\in\nb}$ converges to $f$ in $L^p_{r,q}(D)$ and $(\opa_b f_\nu)_{\nu\in\nb}$  converges to $g$ in $L^p_{r,q+1}(D)$, and we have $g=\opa_{b,s} f$.

Since $L^p_{r,q}(D)$ is a subspace of the space of currents on $D$, the $\opa_b$ operator is defined on $L^p_{r,q}(D)$ in the weak sense, so we can define two operators $\opa_{b,\wt c}~:~L^p_{r,q}(D)\to L^p_{r,q+1}(D)$ and
$\opa_b~:~L^p_{r,q}(D)\to L^p_{r,q+1}(D)$ which coincides with the  $\opa_b$ in the sense of currents and whose domains of definition are  respectively
$$Dom(\opa_{b,\wt c})=\{f\in L^p_{r,q}(M)~|~{\rm supp}~f\subset\ol D, \opa_b f\in L^p_{r,q}(M)\}$$
and
$$Dom(\opa_b)=\{f\in L^p_{r,q}(D)~|~ \opa_b f\in L^p_{r,q}(D)\}.$$

When the boundary of $D$ is sufficiently smooth (at least Lipschitz), we get the following proposition, which is a direct consequence of Friedrichs' Lemma, which holds in $L^p$, $p\geq 1$. Its proof is the same as in the complex case (see lemma 4.3.2 in \cite{ChSh} and lemma 2.4 in \cite{LaShdualiteL2}).
\begin{prop}\label{friedrich}
Let $D$ be a relatively compact domain with Lipschitz boundary in a  CR submanifold $M$ of a complex manifold $X$ and $r,q$ be integers with $0\leq r\leq n$ and $0\leq q\leq n-k$.

(i) A differential form $f\in L^p_{r,q}(D)$ belongs to the domain of definition of the $\opa_b$ operator if and only if there exists a sequence  $(f_\nu)_{\nu\in\nb}$ of forms in $\ec^{r,q}(\ol D)$ such that the sequences $(f_\nu)_{\nu\in\nb}$ and $(\opa f_\nu)_{\nu\in\nb}$ converge respectively to $f$ and $\opa_b f$ in $\|.\|_p$ norm. Therefore $\opa_{b,s}=\opa_b$.

(ii) A differential form $f\in L^p_{r,q}(D)$  belongs to the domain of definition of the $\opa_{b,c}$ operator if and only if the forms $f_0$ and $\opa_b f_0$ are both in $L^p_*(M)$, where $f_0$ denotes the extension by $0$ of $f$ on $M\setminus D$. Therefore $\opa_{b,c}=\opa_{b,\wt c}$.
\end{prop}

We will denote respectively by $H^{r,q}_{L^p}(D)$ and $H^{r,q}_{\wt c,L^p}(D)$ the cohomology groups associated to the complexes $(L^p_{r,\bullet}(D),\opa_b)$ and $(L^p_{r,\bullet}(D),\opa_{b,\wt c})$.

\subsection{Andreotti-Grauert theory in $L^p$ for large degrees}\label{AG}

In this section we will extend to the CR case a part of the results obtained in section 2.2 in \cite{LaLp}.

We will consider the case when $D$ is the intersection of $M$ with a strictly pseudoconvex domain of $X$, more precisely
\begin{defin}
Let $M$  be a $\ci$-smooth, $s$-concave, $s\geq 1$, generic CR submanifold of real codimension $k$ in a complex manifold $X$ of complex dimension $n$. A relatively compact domain $D$ in $M$ is said to be \emph {completely strictly $s$-convex} if there exists a real, $\ci$-smooth, strictly plurisubharmonic function $\rho$ defined on a neighborhood $U$ of $\ol D$ in $X$ such that
$$D=M\cap\{x\in U~|~\rho(x)<0\}.$$
We will say that $D$ is \emph {strictly $s$-convex} if the function $\rho$ is only defined on a neighborhood of the boundary of $D$.
\end{defin}

When $X=\cb^n$, the solvability with $L^p$-estimates of the $\opa_b$-equation in completely strictly $s$-convex domains of a $s$-concave, $s\geq 1$, generic CR submanifold $M$ has been studied in \cite{LaSh}. In particular the following result is proven
\begin{thm}\label{local}
Let $M$  be a $\ci$-smooth, $s$-concave, $s\geq 1$, generic CR submanifold of real codimension $k$ in $\cb^n$ and $D$ a relatively compact strictly $s$-convex domain in $M$ with $\ci$-smooth boundary. If $p\geq 1$, there exists some bounded linear operators $T^r_{q-1}$ from $L^p_{r,q}(D)$ into $L^p_{r,q-1}(D)$, $0\leq r\leq n$, $n-k-s+1\leq q\leq n-k$, such that for any $f\in L^p_{r,q}(D)$ with $\opa_b f=0$ in $D$, we have $\opa_bT^r_{q-1} f=f$.
\end{thm}

Let us go back to the case $X$ is a complex manifold.
\begin{thm}\label{compact}
Let $M$  be a $\ci$-smooth, $s$-concave, $s\geq 1$, generic CR submanifold of real codimension $k$ in a complex manifold $X$ of complex dimension $n$ and $D$ a relatively compact strictly $s$-convex domain in $M$ with $\ci$-smooth boundary. If $p\geq 1$, there exists some bounded linear operators $\wt T^r_{q-1}$ from $L^p_{r,q}(D)$ into $L^p_{r,q-1}(D)$ and compact operators $\wt K^r_q$ from $L^p_{r,q}(D)$ into itself, $0\leq r\leq n$, $n-k-s+1\leq q\leq n-k$, such that for any $f\in L^p_{r,q}(D)$ with $\opa_b f=0$ in $D$, we have
$$f=\opa_b\wt T^r_{q-1} f+\wt K^r_q f.$$
Therefore, for $0\leq r\leq n$ and $n-k-s+1\leq q\leq n-k$, the cohomology groups $H^{r,q}_{L^p}(D)$ are finite dimensional and $\opa_b(L^p_{r,q-1}(D))$ is a closed subspace of $L^p_{r,q}(D)$.
\end{thm}
\begin{proof}
First consider the case when $f$ is $\ci$-smooth up to the boundary. Using an open covering $(U_j)_{j\in J}$ of $\ol D$ and a partition $(\chi_j)_{j\in J}$ of  unity associated to this covering, we can derive a global homotopy formula from the local one in Theorem 2.3.1 in \cite{LaSh}. There exist integral operators $T^r_{q-1}$ and $K^r_{q}$ such that
$$T^r_{q-1} f=\sum_{j\in J} T_j(\chi_j f)\quad {\rm and} \quad K^r_{q} f=\sum_{j\in J} T_j(\opa_b(\chi_j)\wedge f)$$
and
$$f=\opa_b T^r_{q-1} f+T^r_{q}\opa_b f+K^r_{q} f.$$
The local operators $T_j$ are associated to local kernels $R_j$ and $G_j$ in the following way
$$T_j(\varphi)=(-1)^{(r+q)(k+1)+\frac{k(k-1)}{2}}[\int_{D\cap U_j}\varphi(\zeta)\wedge R_j(.,\zeta)+(-1)^{n+q+1}\int_{\pa D\cap U_j}\varphi(\zeta)\wedge G_j(.,\zeta)],$$
if $\varphi$ is a $\ci$-smooth $(r,q)$-form with compact support in $U_j$.
Following section 4 in \cite{LaSh}, we define new kernels $\wt G_j(z,\zeta)$ which are integrable and also their first derivatives in $\zeta$ for each fixed $z$ and moreover coincide with the kernels $G_j$ when $\zeta\in\pa D\cap U_j$ and $z\in D\cap U_j$. Applying Stokes formula we get
\begin{align*}
T_j(\varphi)=(-1)^{(r+q)(k+1)+\frac{k(k-1)}{2}}[\int_{D\cap U_j}\varphi(\zeta)\wedge R_j(.,\zeta)&+(-1)^{r+q+1}\int_{D\cap U_j}\opa_b\varphi(\zeta)\wedge \wt G_j(.,\zeta)\\
&-\int_{D\cap U_j}\varphi(\zeta)\wedge (\opa_b)_\zeta \wt G_j(.,\zeta)],
\end{align*}
if $\varphi$ is a $\ci$-smooth $(r,q)$-form with compact support in $U_j$.

This gives
\begin{equation}\label{homotopie}
f=\opa_b\big (\wt T^r_{q-1} f+\wt S^r_{q-1} \opa_b f\big )+\wt T^r_{q} \opa_b f+\wt K^r_q f+ \wt L^r_{q} \opa_b f,
\end{equation}
with
\begin{align*}
\wt T^r_{q-1} f=&\sum_{j\in J}(-1)^{(r+q)(k+1)+\frac{k(k-1)}{2}}[\int_{D\cap U_j}\chi_j(\zeta)f(\zeta)\wedge R_j(.,\zeta)\\
&+(-1)^{r+q+1}\int_{D\cap U_j}\opa_b\chi_j(\zeta)\wedge f(\zeta)\wedge \wt G_j(.,\zeta)-\int_{D\cap U_j}\chi_j(\zeta)f(\zeta)\wedge (\opa_b)_\zeta \wt G_j(.,\zeta)],
\end{align*}
$$\wt S^r_{q-1} \opa_b f=\sum_{j\in J}(-1)^{k(r+q)+1+\frac{k(k-1)}{2}}\int_{D\cap U_j}\chi_j(\zeta)\opa_b f(\zeta)\wedge \wt G_j(.,\zeta),$$
$$\wt L^r_{q} \opa_b f=\sum_{j\in J}(-1)^{k(r+q)+\frac{k(k-1)}{2}}\int_{D\cap U_j}\opa_b\chi_j(\zeta)\wedge\opa_b f(\zeta)\wedge \wt G_j(.,\zeta)$$
and
\begin{align*}
\wt K^r_{q} f=\sum_{j\in J}(-1)^{(r+q)(k+1)+\frac{k(k-1)}{2}}[&\int_{D\cap U_j}\opa_b\chi_j(\zeta)\wedge f(\zeta)\wedge R_j(.,\zeta)\\
&-\int_{D\cap U_j}\opa_b\chi_j(\zeta)\wedge f(\zeta)\wedge (\opa_b)_\zeta \wt G_j(.,\zeta)].
\end{align*}
It follows from section 4.2 in \cite{LaSh}, that the operators $\wt T^r_{q-1}$  are bounded linear operators from $L^p_{r,q}(D)$ into $L^p_{r,q-1}(D)$, the operators $\wt S^r_{q-1}$ are bounded linear operators from $L^p_{r,q+1}(D)$ into $L^p_{r,q-1}(D)$, the operators $\wt L^r_{q}$ are bounded linear operators from $L^p_{r,q+1}(D)$ into $L^p_{r,q}(D)$ and the operators $\wt K^r_{q}$ are bounded linear operators from $L^p_{r,q}(D)$ into itself since the kernels are respectively of weak type $\frac{2n}{2n-1}$, $\frac{2n+2}{2n}$ and $\frac{2n+2}{2n+1}$.

Assume now that $f\in L^p_{r,q}(D)$ satisfies $\opa_b f=0$ in $D$. By (i) from Proposition \ref{friedrich} there exists a sequence  $(f_\nu)_{\nu\in\nb}$ of forms in $\ec^{r,q}(\ol D)$ such that the sequence $(f_\nu)_{\nu\in\nb}$ converges to $f$ and the sequence $(\opa f_\nu)_{\nu\in\nb}$  converges to $0$ in $L^p$ norm on $D$.
Applying \eqref{homotopie} to the forms $f_\nu$ and letting $\nu$ tends to infinity, the continuity of the operators $\wt T^r_{q-1}$, $\wt S^r_{q-1}$, $\wt L^r_{q}$ and $\wt K^r_{q}$ in $L^p$ norm on $D$ implies
\begin{equation}\label{globalferme}
f=\opa_b\wt T^r_{q-1} f+\wt K^r_q f.
\end{equation}

It remains to prove that the operators $\wt K^r_q$ are compact operators from $L^p_{r,q}(D)$ into itself. Then $Id-\wt K^r_q$ will be a Fredholm operator from $Z^{r,q}_{L^p}(D)$, the space of $\opa_b$-closed $L^p$-forms on $D$, into itself whose range is contained in $\opa_b(L^p_{r,q-1}(D))$ by formula \eqref{globalferme} and therefore the dimension of the cohomology group $H^{r,q}_{L^p}(D)$ will be smaller than the codimension of the range of $Id-\wt K^r_q$ which is finite. The open mapping theorem gives then that $\opa_b(L^p_{r,q-1}(D))$ is a closed subspace of $L^p_{r,q}(D)$.

To prove the compacity of the operators $\wt K^r_q$ from $L^p_{r,q}(D)$ into itself, we will use the estimates of the kernel $(\opa_b)_\zeta G_j(z,\zeta)$ given in section 4.2 of \cite{LaSh} and follow the proof of Lemma 2.20 in \cite{Ho&al}. Since $\wt K^r_{q}=\sum_{j\in J}(-1)^{(r+q)(k+1)+\frac{k(k-1)}{2}+1} K_j$, with
$$K_jf=\int_{D\cap U_j}\opa_b\chi_j(\zeta)\wedge f(\zeta)\wedge \big (R_j(.,\zeta)-(\opa_b)_\zeta \wt G_j(.,\zeta)\big )$$
if $f$ is a $\opa_b$-closed form belonging to $L^p_{r,q}(D)$, it is sufficient to prove the compacity of each operator $K_j$ from $L^p_{r,q}(D)$ into itself.

After a good choice of coordinates the coefficients of the kernel $R_j(z,\zeta)$ and $(\opa_b)_\zeta G_j(z,\zeta)$, whose singularities are concentrated on the diagonal, are bounded by $H_s(\zeta-z)$, with $1\leq s\leq k$ for $R_j$ and $1\leq s\leq k+2$ for $(\opa_b)_\zeta G_j$, where
\begin{equation}\label{controle}
H_s(t)=\frac{1}{\Pi_{i=1}^s(|t_i|+|t|^2)|t|^{2n-k-s-1}}.
\end{equation}

Following \cite{Ho&al}, for any $\varepsilon>0$, we will decompose $K_j$ in two parts, $K_j=K_j^\varepsilon +K_j^b$ where $K_j^\varepsilon$ and $K_j^b$ are defined by
$$\wt K_j^b f(z)=\int_{D\cap\{|z-\zeta|>\varepsilon\}} \varphi(\zeta)\wedge \big (R_j(.,\zeta)-(\opa_b)_\zeta \wt G_j(.,\zeta)\big )$$
and
$$\wt K_j^\varepsilon f(z)=\int_{D\cap\{|z-\zeta|<\varepsilon\}} \varphi(\zeta)\wedge\big (R_j(.,\zeta)-(\opa_b)_\zeta \wt G_j(.,\zeta)\big )$$
for $\varphi$ with compact support in $U_j$.

Note that the kernel $R_j(z,\zeta)-(\opa_b)_\zeta \wt G_j(z,\zeta)$ is bounded on $D\cap U_j\cap\{|z-\zeta|>\varepsilon\}$ and therefore $K_j^b$ is Hilbert-Schmidt and hence compact.

We will prove that the norm of the operator $K_j^\varepsilon$ tends to zero when $\varepsilon$ tends to zero, which will implies that $K_j$ is a compact operator. By Young's inequality and the control of the kernels $R_j(z,\zeta)$ and $(\opa_b)_\zeta G_j(z,\zeta)$ by $H_s(\zeta-z)$, $1\leq s\leq k+2$, with $H_s$ defined by \eqref{controle}, it is sufficient to prove that $\int_{|t|<\varepsilon} H_s(t) dt_1\dots dt_{2n-k}$ tends to zero when $\varepsilon$ tends to zero, if $1\leq s\leq k+2$. But as $1\leq s\leq\frac{n-k}{2}$ and $1\leq s\leq k+2$, we have $2n-k-s>0$ and if $t'=(t_{s+1},\dots,t_{2n-k})$, then
$$\int_{|t|<\varepsilon} H_s(t) dt_1\dots dt_{2n-k}\leq \int_{|t|<\varepsilon} \frac{( \ln|t'|)^sdt_{s+1}\dots dt_{2n-k}}{|t'|^{2n-k-s-1}}\leq C\int_0^\varepsilon (\ln r)^sdr=\delta(\varepsilon),$$
using polar coordinates. Since the function $(\ln r)^s$ is integrable nearby the origine, $\delta(\varepsilon)\to 0$ when $\varepsilon\to 0$, which ends the proof.
\end{proof}

Exactly as in \cite{LaLp} the Grauert's bumping method can be developed in the $L^p$ setting for CR manifolds and it follows from Theorems \ref{local} and \ref{compact} that
\begin{thm}\label{isomLp}
Let $M$  be a $\ci$-smooth, $s$-concave, $s\geq 1$, generic CR submanifold of real codimension $k$ in a complex manifold $X$ of complex dimension $n$ and $D$ a relatively compact strictly $s$-convex domain in $M$ with $\ci$-smooth boundary. If $p\geq 1$, for any pair $(r,q)$ with $0\leq r\leq n$ and $n-k-s+1\leq q\leq n-k$,
$$H^{r,q}_{L^p}(D)\sim H^{r,q}_{L^p_{loc}}(D).$$
\end{thm}
Note that associating Theorem \ref{isomgd} to Theorem \ref{isomLp}, we get that $H^{r,q}_\infty(D)=0$ implies $H^{r,q}_{L^p}(D)=0$.

As in the complex case (see Corollary 2.11 in \cite{LaLp}), using the previous remark, Theorem 6.1 in \cite{HiNa1} and Theorem \ref{local},  we get a vanishing theorem when $D$ is completely strictly $s$-convex.

\begin{cor}
Let $M$  be a $\ci$-smooth, $s$-concave, $s\geq 1$, generic CR submanifold of real codimension $k$ in a complex manifold $X$ of complex dimension $n$ and $D$ a relatively compact completely strictly $s$-convex domain in $M$ with $\ci$-smooth boundary. If $p\geq 1$, for any pair $(r,q)$ with $0\leq r\leq n$ and $n-k-s+1\leq q\leq n-k$,
$$H^{r,q}_{L^p}(D)=0.$$
More precisely, if  $0\leq r\leq n$ and $n-k-s+1\leq q\leq n-k$, there exists a continuous linear operator $T$ from the Banach space $Z^{r,q}_{L^p}(D)$ of $\opa_b$-closed $(r,q)$-forms with coefficients in $L^p(D)$ into $L^p_{r,q-1}(D)$ such that
$$\opa_b T f=f\quad {\rm for~ any}~ f\in Z^{r,q}_{L^p}(D).$$
\end{cor}

\subsection{Weak Cauchy problem in $L^p$ for the tangential Cauchy-Riemann operator}

In this section we will study some Cauchy problem. Let  $M$  be a $\ci$-smooth, generic CR submanifold of real codimension $k$ in a complex manifold $X$ of complex dimension $n$ and $D$ a relatively compact domain in $M$ with $\ci$-smooth boundary. For $p\geq 1$, given any $(r,q)$-form $f$ with coefficients in $L^p(M)$,  $0\leq r\leq n$ and $1\leq q\leq n-k$, such that
$${\rm supp}~f\subset\ol D \quad {\rm and}\quad \opa_b f=0~{\rm in~the~weak~sense~in~}M,$$
does there exists a $(r,q-1)$-forme $g$  with coefficients in $L^p(M)$ such that
$${\rm supp}~g\subset\ol D \quad {\rm and}\quad \opa_b g=f~{\rm in~the~weak~sense~in~}M~?$$
A positive answer to this problem is equivalent to the vanishing of the cohomology group $H^{r,q}_{\wt c,L^p}(D,E)$.

We will solve this problem using Serre duality. First observe that if $p,p'\in\rb$ satisfy $p>1$ and $\frac{1}{p}+\frac{1}{p'}=1$ and $r\in\nb$ satisfy $0\leq r\leq n$, the complexes $(L^p_{r,\bullet}(D),\opa_b)$ and $(L^{p'}_{n-r,\bullet}(D),\opa_{b,\wt c})$ are dual complexes.
\begin{prop}
Let $D$ be a relatively compact domain in $M$ with $\ci$-smooth boundary.

(i) The transpose map of the $\opa_{b,\wt c}$ operator from $L^p_{r,q}(D)$ into $L^p_{r,q+1}(D)$ is the operator $\opa_b=\opa_{b,s}$ from $L^{p'}_{n-r,n-k-q-1}(D)$ into $L^{p'}_{n-r,n-k-q}(D)$.

(ii) The transpose map of the $\opa_b$ operator from $L^p_{r,q}(D)$ into $L^p_{r,q+1}(D)$ is the operator $\opa_{\wt c}=\opa_{b,c}$ from $L^{p'}_{n-r,n-k-q-1}(D)$ into $L^{p'}_{n-r,n-k-q}(D)$.
\end{prop}
\begin{proof}
The proof is the same as in the complex case (see Proposition 2.13 in \cite{LaLp}).
\end{proof}

Let us recall the following abstract result on duality (see section 2.3 in \cite{LaLp} for more details).
\begin{prop}\label{dual}
Let $(E^\bullet,d)$ and $(E'_\bullet,d')$ be two dual complexes of reflexive Banach spaces with densely defined unbounded operators. Assume that $H_q(E'_\bullet)$ is Hausdorff and $H_{q+1}(E'_\bullet)=0$, then $H^{q+1}(E^\bullet)=0$.
\end{prop}

Let $p>1$ be a real number and $p'$ such that $\frac{1}{p}+\frac{1}{p'}=1$. Applying Proposition \ref{dual} to the complex $(E^\bullet,d)$ with, for some fixed $r$ such that $0\leq r\leq n$, $E^q=L^{p}_{r,q}(D)$ if $0\leq q\leq n-k$ and $E^q=\{0\}$ if $q<0$ or $q>n-k$, and $d=\opa_{\wt c}$, we get
\begin{thm}
Let  $M$  be a $\ci$-smooth, generic CR submanifold of real codimension $k$ in a complex manifold $X$ of complex dimension $n$ and $D$ a relatively compact domain in $M$ with $\ci$-smooth boundary. Let $p$ and $p'$ be two real numbers such that $p>1$ and $\frac{1}{p}+\frac{1}{p'}=1$ and $q$ an integer such that $1\leq q\leq n-k$. Assume that $H^{n-r,n-k-q+1}_{L^{p'}}(D)$ is Hausdorff and $H^{n-r,n-k-q}_{L^{p'}}(D)=0$, then
$$H^{r,q}_{\wt c,L^p}(D)=0.$$
\end{thm}

Using the finiteness and vanishing theorems from section \ref{AG}, this gives
\begin{cor}\label{cauchy}
Let $M$  be a $\ci$-smooth, $s$-concave, $s\geq 2$, generic CR submanifold of real codimension $k$ in a complex manifold $X$ of complex dimension $n$ and $D$ a relatively compact completely strictly $s$-convex domain in $M$ with $\ci$-smooth boundary. If $p>1$, for any pair $(r,q)$ with $0\leq r\leq n$ and $1\leq q\leq s-1$, given any $(r,q)$-form $f$ with coefficients in $L^p(M)$ such that
$${\rm supp}~f\subset\ol D \quad {and}\quad \opa_b f=0~{in~the~weak~sense~in~}M,$$ there exists a $(r,q-1)$-forme $g$  with coefficients in $L^p(M)$ such that
$${\rm supp}~g\subset\ol D \quad {and}\quad \opa_b g=f~{in~the~weak~sense~in~}M.$$
\end{cor}
The case $q=1$ was already settled in Proposition \ref{d1} and Theorem
\ref{resolcomp1} under the assumptions $M$ is 1-concave,
$H^{r,1}_{c,cur}(M)=0$  and $M\setminus D$ connected or
$<f,\varphi>=0$ for all $\varphi\in Z^{r,n-k-1}_\infty(U)$ for any
neighborhood $U$ of $\ol D$. 

\providecommand{\bysame}{\leavevmode\hbox to3em{\hrulefill}\thinspace}
\providecommand{\MR}{\relax\ifhmode\unskip\space\fi MR }
\providecommand{\MRhref}[2]{%
  \href{http://www.ams.org/mathscinet-getitem?mr=#1}{#2}
}
\providecommand{\href}[2]{#2}

\enddocument

\end